\documentclass{amsart} \usepackage{hyperref}
\pdfoutput=1
\usepackage{enumerate} 
\usepackage{upref}
\usepackage{verbatim} 
\usepackage{color}
\usepackage{graphicx}
\newcommand{\sign}{\mathop{\rm sign}}
\newcommand*{\mailto}[1]{\href{mailto:#1}{\nolinkurl{#1}}}
\usepackage[latin1]{inputenc}
\usepackage{amssymb, amsmath, amsfonts, amsthm}
\usepackage[all]{xy}

\DeclareMathOperator{\id}{Id}
\DeclareMathOperator{\meas}{meas}

\newcommand{\dott}{\, \cdot\,}

\newcommand{\quot}{{\F/\Gr}}
\newcommand{\Gr}{G}

\newcommand{\I}{\mathrm{i}}
\newcommand{\D}{\ensuremath{\mathcal{D}}}
\newcommand{\G}{\ensuremath{\mathcal{G}}}
\newcommand{\F}{\ensuremath{\mathcal{F}}}

\newcommand{\inv}{{^{-1}}}
\newcommand{\abs}[1]{\left\vert#1\right\vert}

\newcommand{\Real}{\mathbb R}

\newcommand{\norm}[1]{\left\Vert#1\right\Vert}

\DeclareMathOperator{\sgn}{sgn}
\newcommand{\nn}{\nonumber}

\newtheorem{theorem}{Theorem}[section]

\newtheorem{lemma}[theorem]{Lemma}
\newtheorem{definition}[theorem]{Definition}

\numberwithin{equation}{section}

\allowdisplaybreaks

\begin{document}

\title[Global solutions with nonvanishing
asymptotics]{Global conservative solutions of the
  Camassa--Holm equation for initial data with
  nonvanishing asymptotics}

\author[K. Grunert]{Katrin Grunert}
\address{Faculty of Mathematics\\ University of
  Vienna\\ Nordbergstrasse 15\\ A-1090 Wien\\
  Austria}
\email{\mailto{katrin.grunert@univie.ac.at}}
\urladdr{\url{http://www.mat.univie.ac.at/~grunert/}}

\author[H. Holden]{Helge Holden}
\address{Department of Mathematical Sciences\\
  Norwegian University of Science and Technology\\
  7491 Trondheim\\ Norway\\ {\rm and} Centre of
  Mathematics for Applications\\ University of Oslo\\
  NO-0316 Oslo\\ Norway}
\email{\mailto{holden@math.ntnu.no}}
\urladdr{\url{http://www.math.ntnu.no/~holden/}}

\author[X. Raynaud]{Xavier Raynaud}
\address{Centre of Mathematics for Applications\\
  University of Oslo\\ NO-0316 Oslo\\ Norway}
\email{\mailto{xavierra@cma.uio.no}}
\urladdr{\url{http://folk.uio.no/xavierra/}}

\date{\today} 
\thanks{Research supported in part by the
  Research Council of Norway under projects Wavemaker, NoPiMa, and by the Austrian Science Fund (FWF) under Grant No.~Y330.}  
\subjclass[2010]{Primary:
  35Q53, 35B35; Secondary: 35Q20}
\keywords{Camassa--Holm equation, conservative solutions, nonvanishing asymptotics}

\begin{abstract}
  We study global conservative solutions of the
  Cauchy problem for the Camassa--Holm equation
  $u_t-u_{txx}+\kappa
  u_x+3uu_x-2u_xu_{xx}-uu_{xxx}=0$ with
  nonvanishing and distinct spatial asymptotics.
\end{abstract}
\maketitle

\section{Introduction}

The Cauchy problem for the Camassa--Holm (CH)
equation \cite{CH, CHH},
\begin{equation}
  \label{eq:kappaCH}
  u_t-u_{txx}+\kappa u_x+3uu_x-2u_xu_{xx}-uu_{xxx}=0,
\end{equation}
where $\kappa\in\Real$ is a constant, has
attracted much attention due to the fact that it
serves as a model for shallow water waves
\cite{cla} and its rich mathematical
structure. For example, it has a bi-Hamiltonian
structure, %\cite{ff}
infinitely many conserved quantities %\cite{L}
and blow-up phenomena have been studied. As these
properties play no role in the present approach,
we refer to \cite{HR} for references to papers
that discuss these properties.
%\cite{cons:98}, \cite{cons:98b}, and \cite{cons:00}. 

In particular, global conservative solutions have
been constructed in the periodic case \cite{HR3}
and on the real line in the case of initial data
with the same vanishing asymptotics at minus and
plus infinity \cite{BC,HR} (for $\kappa=0$) and
\cite{HR2} (for $\kappa\neq0$).  Furthermore, a
Lipschitz metric has been derived for the
Camassa--Holm equation \cite{GHR,GHR2}.

Here we focus on the construction of a semigroup
of global conservative solutions on the real line
for initial data with (in general) different
asymptotics at minus and plus infinity.  The
approach used here is similar to the one used in
\cite{HR} for the Camassa--Holm equation in the
case of vanishing asymptotics (when $u\in
H^1(\Real)$). It also resembles a recent study of
the Hunter--Saxton equation in \cite{BHR}, and
indeed we here combine the two approaches. In
\cite{bencockar:06}, the authors proves in the
dissipative case the existence of $H^1$
perturbations around a given solution with
nonvanishing asymptotics. In this article, by
solving the Cauchy problem, we prove that such
solutions exist in the conservative case.

More precisely, we consider the initial problem
for \eqref{eq:kappaCH} with initial data $u_0\in
H_\infty(\Real)$, which means that $u_0$ can be
written as
\begin{equation}
 u_0(x)=\bar u_0(x)+c_-\chi(-x)+c_+\chi(x),
\end{equation}
for some constants $c_\pm\in\Real$, where $\bar
u_0\in H^1(\Real)$ and $\chi$ denotes a smooth
partition function, which satisfies $\chi(x)=0$
for $x\leq 0$, $\chi(x)=1$ for $x\geq 1$ and
$\chi'(x)\geq 0$ for all $x\in\Real$.

In \cite{L1, L2}, traveling waves solutions of the
CH equation have been characterized and
depicted. The solutions are obtained by glueing together simpler solutions. 
In particular, Lenells constructs solutions with
distinct asymptotics at plus and minus infinity,
see Figure \ref{fig:tw} for an example.
\begin{figure}[h]
  \centering
  \includegraphics[width=10cm]{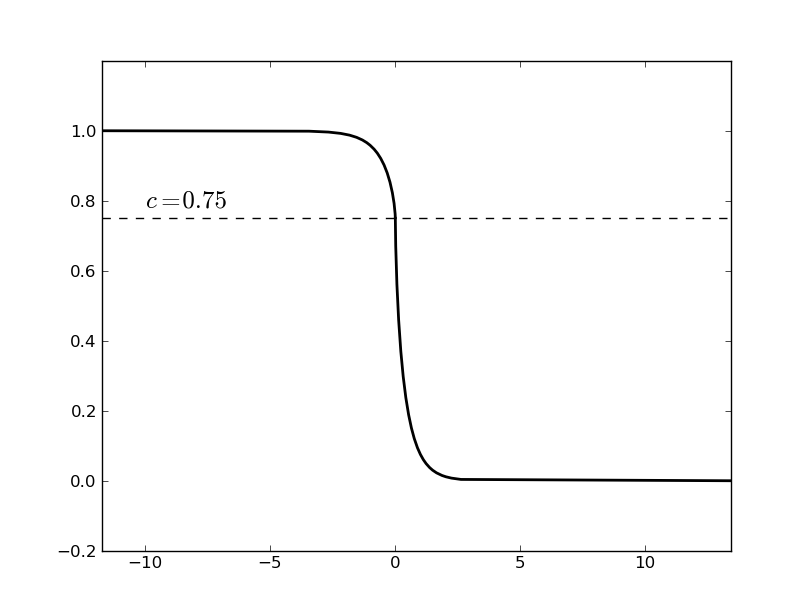}
  \caption{Traveling wave $u(t,x)=\phi(x-ct)$
    obtained by gluing together two \textit{cuspons} at the
    point $(0,c)$. For the picture here, we use
    $c=\frac34$ and $k=-\frac78$ so that the
    gluing indeed results in a weak solution (see
    \cite{L1}). Note that $\phi_x\in L^2(\Real)$ but
    $\phi$ is not Lipschitz as
    $\lim_{x\to0}\phi_x(x)=-\infty$.}
  \label{fig:tw}
\end{figure}

Without loss of generality, we can consider the
case $\kappa=0$ (see Section
\ref{sec:euler}). Then \eqref{eq:kappaCH} can be
rewritten as the following system
\begin{subequations}\label{eq:intro1}
\begin{align}
 u_t+uu_x+P_x&=0,\label{eq:intro1A} \\
 P-P_{xx}&=u^2+\frac{1}{2}u_x^2. \label{eq:intro1B}
\end{align}
\end{subequations}
The last equation yields, after using the Green
function,
\begin{equation*}
  P(x)=\frac12\int_\Real e^{-\abs{x-z}}(u^2+\frac12u_x^2)(z)\,dz
\end{equation*}
and we can see that $P$ is well-defined when $u\in
H_\infty(\Real)$. We  show that the asymptotic
values $c_-$ and $c_+$ for $u(t,x)$ at infinity are
in fact constant in time (see
Section \ref{sec:euler}). This may seem a little
surprising since the Camassa--Holm equation
with vanishing asymptotics has
infinite speed of propagation \cite{cons:05}, and therefore one
might think that the solution would approach a common 
asymptotic, e.g., the mean value, at plus and minus infinity.

The solutions of the Camassa--Holm equation may
experience wave breaking (see, e.g.,
\cite{cons:98b} and references therein). The
continuation of the solution past wave breaking is
highly nontrivial, and allows for several distinct
continuations. Two prominent classes of solutions
are denoted as conservative and dissipative
solutions, see, e.g.,
\cite{BC,BreCons:05a,HR,HRa}.

The aim of this article is to construct a
semigroup of conservative solutions on the line
for nonvanishing asymptotics. In the case of
vanishing asymptotics, conservative solutions
refer to the solutions for which the $H^1(\Real)$ norm is
preserved, for almost every time. Here, it does
not make sense to consider the $H^1(\Real)$ norm of the
solution. By conservative solutions, we mean weak
solutions to the equation which in addition
satisfy the conservation law
\begin{equation*}
  (u^2+u_x^2)_t+(u(u^2+u_x^2))_x=(u^3-2Pu)_x
\end{equation*}
in the sense of  distributions (see Definition
\eqref{eq:defweakconssol} for the precise
definition). 

The continuity of the semigroup is obtained with
respect of a new metric that we introduce. This
metric depends on the choice of the partition
function $\chi$. However, in Section~\ref{top1},
it is shown that different choices of partition
functions $\chi$ lead to the same topology.

\section{Eulerian coordinates} \label{sec:euler}

We consider the Cauchy problem for the Camassa--Holm equation with arbitrary 
$\kappa\in\Real$ given by
\begin{equation}
\label{eq:chfull}
u_t-u_{txx}+\kappa u_x+3uu_x-2u_xu_{xx}-uu_{xxx}=0, \quad u|_{t=0}=u_0.
\end{equation}
We are interested in global solutions for initial data with nonvanishing
limits at infinity, that is, 
\begin{equation}
  \label{eq:nonvanlim}
  \lim_{x\to-\infty} u_0(x)=u_{-\infty}\quad\text{ and }\quad\lim_{x\to\infty} u_0(x)=u_{\infty}.
\end{equation}
To be more specific, we assume that $u_0(x)$ can be rewritten as 
\begin{equation}
  \label{eq:u0withas}
 u_0(x)=\bar u_0(x)+u_{-\infty}\chi(-x)+u_\infty \chi(x),
\end{equation}
with $\bar u_0\in H^1(\Real)$ and $\chi$ a smooth
partition function, which has support in
$[0,\infty)$, satisfies $\chi(x)=1$ for $x\geq 1$
and $\chi'(x)\geq 0$ for $x\in\Real$. Assuming
that $u(t,x)$ satisfies \eqref{eq:kappaCH}, then
the function $v(t,x)=u(t,x-\kappa t/2)+\kappa/2$
satisfies \eqref{eq:kappaCH} with $\kappa=0$ and
hence we can without loss of generality assume
that $\kappa=0$, because the framework presented
here allows for nonvanishing asymptotics.

Let us introduce the mapping $I_{\chi}$ from
$H^1(\Real)\times\Real^2$ into
$H_{\text{loc}}^1(\Real)$ given by
\begin{equation*}
  I_{\chi}(\bar u,c_-,c_+)=\bar u(x)+c_-\chi(-x)+c_+\chi(x)
\end{equation*}
for any $(\bar u,c_-,c_+)\in
H^1(\Real)\times\Real^2$. We denote by
$H_{\infty}(\Real)$ the image of
$H^1(\Real)\times\Real^2$ by $I_\chi$, that is,
$H_{\infty}(\Real)=I_{\chi}(H^1(\Real)\times\Real^2)$. Then,
$u_0$ as given by \eqref{eq:u0withas} belongs by
definition to $H_\infty(\Real)$. Since $I_\chi$ is
linear, $H_{\infty}(\Real)$ is a vector space. The
mapping $I_\chi$ is injective. We equip
$H_\infty(\Real)$ with the norm
\begin{equation}
  \label{eq:normHinfty}
  \norm{u}_{H_\infty(\Real)}=\norm{\bar u}_{H^1(\Real)}+\abs{c_-}+\abs{c_+}
\end{equation}
where $u=I_\chi(\bar u,c_-,c_+)$.  Then,
$H_\infty(\Real)$ is a Banach space. Given another
partition function $\tilde\chi$, we define the
mapping $(\tilde{\bar u},\tilde c_-,\tilde
c_+)=\Psi(\bar u,c_-,c_+)$ from
$H^1(\Real)\times\Real^2$ to
$H^1(\Real)\times\Real^2$ as $\tilde c_-=c_-$,
$\tilde c_+=c_+$ and
\begin{equation}
  \label{eq:defPsi}
  \tilde{\bar u}(x)=\bar u(x)+c_-(\chi(-x)-\tilde\chi(-x))+c_+(\chi(x)-\tilde\chi(x)).
\end{equation}
The linear mapping $\Psi$ is a continuous
bijection. Since
\begin{equation*}
  I_{\chi}=I_{\tilde\chi}\circ \Psi,
\end{equation*}
we can see that the definition of the Banach space
$H_\infty(\Real)$ does not depend on the choice of
the partition function $\chi$. The norm defined by
\eqref{eq:normHinfty} for different partition
functions $\chi$ are all equivalent.

If $u(t,x)$ is a solution of the Camassa--Holm
equation, then, for any constants $\alpha$ and
$\beta$, we easily find that
\begin{equation}
  \label{eq:rescale}
  v(t,x)=\alpha u(\alpha t, x-\beta t)+\beta
\end{equation}
is also a solution with $\kappa$ replaced by $\alpha\kappa-2\beta$. The case $u_\infty= u_{-\infty}$, can be reduced to the standard case of vanishing asymptotics at infinity by choosing $\alpha=1$ and $\beta=-u_\infty$.
Furthermore, in the case when $u_\infty\neq u_{-\infty}$, it is no loss of generality to only consider initial conditions which satisfy the following conditions
\begin{equation}
 \label{eq:limvalue2}
  \lim_{x\to-\infty}u_0(x)=0\quad \text{and }\quad \lim_{x\to\infty}u_0(x)=c,
\end{equation}
where $c$ denotes some constant. 
Especially, if $u_0$ is an initial condition which satisfies 
\eqref{eq:nonvanlim}, then take 
$\alpha=1$ and $\beta=-u_{-\infty}$ in 
\eqref{eq:rescale} and we obtain an initial condition $v_0$ which satisfies condition \eqref{eq:limvalue2}.
Accordingly, we introduce the subspace $H_{0,\infty}(\Real)$ of $H_\infty(\Real)$ as 
\begin{equation*}
  H_{0,\infty}(\Real)=I_{\chi}(H^1(\Real)\times\{0\}\times\Real).
\end{equation*}
As we shall see next this subspace is preserved by
the Camassa--Holm equation. This is the main
motivation for considering this space beside of
the fact that the arguments simplify.

In what follows we will restrict ourselves to the
case $\kappa=0$, as the case $\kappa\not =0$ can
be treated using the same ideas and
techniques. For the case $\kappa=0$ the governing
equations read\footnote{For $\kappa$ nonzero
  \eqref{eq:p} is simply replaced by
  $P-P_{xx}=u^2+\kappa u +\frac{1}{2}u_x^2$.}
\begin{subequations}
\label{eq:chf}
\begin{align}
\label{eq:ch}
&u_t+uu_x+P_x=0,\\
\label{eq:p}
&P-P_{xx}=u^2+\frac{1}{2}u_x^2.
\end{align}
\end{subequations}

Let us assume that $\lim_{x\to\infty}u(t,x)=c(t)$
exists. From \eqref {eq:p} we obtain the following
representation for $P$, under the assumption that
$u\in H_{0,\infty}(\Real)$,
\begin{equation}\label{rep:p}
 P(x)=c^2\chi^2(x)+\frac{1}{2}\int_\Real e^{-\vert x-y\vert } (2c\chi\bar u +\bar u^2+\frac{1}{2}u_x^2+2c^2\chi^{\prime 2}+2c^2\chi\chi^{\prime\prime})(y) dy.
\end{equation}
It follows that
\begin{equation*}
  P_x(x)=2c^2\chi'(x)\chi(x)-\frac12\int_\Real\sgn{(x-y)}e^{-\abs{x-y}}(2c\chi\bar u +\bar u^2+\frac{1}{2}u_x^2+2c^2\chi^{\prime 2}+2c^2\chi\chi^{\prime\prime})(y) dy
\end{equation*}
and $\lim_{x\to\infty}P_x=0$. Thus, we formally
obtain that $c'(t)=\lim_{x\to\infty}u_t(t,x)=0$
and the limit at infinity of $u$ is constant in
time.  Indeed the solutions we are going to
construct satisfy this property.

% In the next
% section we will prove rigorously that we have
% $u(t,x)=\bar u(t,x)+c\chi(x)$.  , according to
% \eqref{eq:ch}
%\begin{equation*}
% \bar u_t(t,x)+c^\prime(t)\chi(x)\in L^2(\Real),
%\end{equation*}
%which implies that $c^\prime(t)=0$ and therefore
%we can write $u(t,x)=\bar
%u(t,x)+c\chi(x)$. 

\section{Lagrangian variables}
\label{sec:lagvar}

The aim of this section is to rewrite the
Camassa--Holm equation as a system of ordinary
differential equations, which describes solutions
in Lagrangian coordinates. Let $V$ be the Banach
space
\begin{equation*}
  V=\{f\in C_b(\Real) \mid  f_\xi\in L^2(\Real)\},
\end{equation*}
where  $C_b(\Real)=C(\Real)\cap L^\infty(\Real)$, equipped with the norm 
\begin{equation}
 \norm{f}_V=\norm{f}_{L^\infty}+\norm{f_\xi}_{L^2},
\end{equation}
for any $f\in V$.  Then we define
$y_t(t,\xi)=u(t,y(t,\xi))$, which can be rewritten
as $y(\xi)=\zeta(\xi)+\xi$, where $\zeta$ belongs
to $V$. Furthermore we set
$U(t,\xi)=u(t,y(t, \xi))$, which can be decomposed
as
\begin{equation*}
  U=\bar U+c\chi\circ y 
\end{equation*}
where $\bar U\in H^1(\Real)$ and $c\in\Real$. We
define $h \in L^2(\Real)$ formally as
\begin{equation*}
  h(t,\xi)=u_x^2(t,y(t,\xi))y_\xi(t,\xi)
\end{equation*}
so that $u_x^2(t,x)\,dx=y_\#(h(t,\xi)\,d\xi)$. Here $y_\#$ denotes the 
push-forward.\footnote{The push-forward of a measure $\nu$ by a measurable function $f$ is the measure $f_\#\nu$ defined as $f_\#\nu(B)=\nu(f^{-1}(B))$ for any Borel set $B$.} We
have
\begin{equation}
  \label{eq:energyxi}
  y_\xi h=U_\xi^2.
\end{equation}
Next we derive the equivalent system for the independent variables $\zeta$, $U$, and $h$. Therefore set $P(t,\xi)=P(t,y(t,\xi))$, where $P(t,x)$ is the solution of \eqref{eq:p} and define $Q(t,\xi)=P_x(t,y(t,\xi))$, then we have using \eqref{eq:ch}
\begin{align*}
  y_t&=U,\\
  U_t&=-Q.
\end{align*}
Let us compute $h_t$. Assuming that the solution
is smooth, \eqref{eq:chf} yields
\begin{equation}
  \label{eq:trenerg}
  (u_x^2)_t+(uu_x^2)_x=2(u^2u_x-Pu_x).
\end{equation}
By the definition of $y$ as the characteristic
function, it follows that
\begin{equation}
  \label{eq:derh}
  \frac{\partial}{\partial t}(u_x^2\circ yy_\xi)=((u_x^2)_t+(uu_x^2)_x)\circ yy_\xi.
\end{equation}
Hence, from \eqref{eq:trenerg}, we get
\begin{equation}
  \label{eq:defht}
  h_t=2((u^2-P)\circ y)(u_x\circ y)y_\xi=2(U^2-P)U_\xi.
\end{equation}
Thus, we consider the system
\begin{subequations}
  \label{eq:govsys}
  \begin{align}
    \label{eq:govsys1}
    y_t&=U,\\
    \label{eq:govsys2}
    U_t&=-Q,\\
    \label{eq:govsys3}
    h_t&=2(U^2-P)U_\xi.
  \end{align}
\end{subequations}
After studying the functions $P$ and $Q$ we will prove the local existence of solutions to
\eqref{eq:govsys} in
\begin{equation*}
  E:=V\times H_{0,\infty}(\Real)\times L^2(\Real)
\end{equation*}
by using a standard contraction argument. 
The norm
of $E$ is given in terms of a partition function
$\chi$. Then $E$ is in isometry with
\begin{equation*}
  \bar E=V\times H^1(\Real)\times\Real\times L^2(\Real).
\end{equation*}
We have
\begin{equation*}
  \norm{(\zeta,U,h)}_{E}=\norm{(\zeta,I_\chi^{-1}(U),h)}_{\bar E}.
\end{equation*}
However, as noted earlier, all partition functions
give rise to equivalent norms. For convenience, we
will often abuse the notations and denote by the
same $X$ the two elements $(\zeta,\bar U,c,h)$ and
$(y,U,h)$ where, by definition, $U=\bar
U+c\chi\circ y$ and $y(\xi)=\zeta(\xi)+\xi$.

\begin{lemma}
  \label{lem:PQ}
  For any $X=(\zeta,U,h)$ in $E$, we define the
  maps $\mathcal{Q}$ and $\mathcal{P}$ as
  $\mathcal{Q}(X)=Q$ and $\mathcal{P}(X)=P$, where
  $P$ and $Q$ are given by 
\begin{equation}\label{eq:defP}
 P(\xi)=\frac14\int_\Real
  e^{-\abs{y(\xi)-y(\eta)}}((2\bar U^2+4c\bar
  U\chi\circ y)y_\xi+h)(\eta)\,d\eta+c^2g\circ y(\xi),
\end{equation}
 and
 \begin{equation}\label{eq:defQ}
   Q(\xi)=-\frac14\int_\Real \sign(\xi-\eta)e^{-\abs{y(\xi)-y(\eta)}}((2\bar
  U^2+4c\bar U\chi\circ y)y_\xi+h)(\eta)\,d\eta+c^2 g'\circ y(\xi),
 \end{equation}
where 
\begin{equation}\label{eq:defG}
 g(x)=\chi^2(x)+\frac12\int_\Real e^{-\abs{x-z}}(2\chi^{\prime 2}+2\chi\chi^{\prime\prime})(z) \,dz.
\end{equation}
Then, $X\mapsto\mathcal{P}-U^2$ and
$X\mapsto\mathcal{Q}$ are locally Lipschitz maps, i.e.,
Lipschitz on bounded sets, from $E$ to
$H^1(\Real)$. Moreover we have
  \begin{equation}\label{eq:derQ}
    Q_\xi= -\frac12 h-(U^2-P) y_\xi,
  \end{equation}
  \begin{equation}\label{eq:derP}
    P_\xi=Q(1+\zeta_\xi).
  \end{equation}
\end{lemma}

\begin{proof}
  The expression \eqref{eq:defP} is obtained from
  \eqref{eq:p} after a change of variables to
  Lagrangian variables. From \eqref{eq:p}, we get
\begin{equation*}
  P-P_{xx}=c^2\chi^2+2c\chi\bar u+\bar u^2+\frac{1}{2}u_x^2,
\end{equation*}
which yields, after applying the Helmholtz
operator,
\begin{equation}
  \label{eq:intdefP}
  P(x)= c^2g(x)+\frac{1}{2}\int_\Real
  e^{-\abs{x-z}}(2c\chi\bar u+\bar
  u^2+\frac{1}{2}u_x^2)(y)dy,
\end{equation}
where we define $g$ as the solution of
$g-g^{\prime\prime}=\chi^2$. Since
$(g-\chi^2)-(g-\chi^2)_{xx}=2\chi^{\prime\prime}\chi+2\chi^{\prime
  2}$, after applying the Helmholtz operator, we get
\begin{equation}
  \label{eq:defgint}
  g-\chi^2=\frac12\int_\Real e^{-\abs{x-z}}(2\chi^{\prime 2}+2\chi\chi^{\prime\prime})(z) \,dz
\end{equation}
and we recover the definition \eqref{eq:defG}.
The integral term on the right-hand side in
\eqref{eq:defgint} belongs to $H^1(\Real)$, and
thus it follows that $\lim_{x\to -\infty}g(x)=0$
and $\lim_{x\to\infty} g(x)=1$. Moreover, since we
can also write $g(x)=\frac12\int_\Real
e^{-\abs{z}}\chi^2(x-z)\,dz$, we have
\begin{equation}\label{eq:gprime}
  g'(x)=\frac12\int_\Real e^{-\abs{x-z}}2\chi'(z)\chi(z)\,dz
\end{equation}
so that $g'\geq 0$, as $\chi'\geq0$. Thus,
\begin{equation*}
  \norm{g}_{L^\infty}=\lim_{x\to\infty} g(x)=1.
\end{equation*}
Defining now $P(t,\xi)=P(t,y(t,\xi))$, then
\eqref{eq:intdefP} yields \eqref{eq:defP}, after
changing variable and using
\eqref{eq:energyxi}. Analogously one explains the
definition \eqref{eq:defQ}.

Next we prove that $\mathcal{Q}$ is locally
Lipschitz from $E$ to $H^1(\Real)$. We rewrite
$\mathcal{Q}$ as
  \begin{align*}
    \mathcal{Q}(X)(\xi)
    & = -\frac{e^{-\zeta(\xi)}}{4}\int_{-\infty}^\xi e^{-\vert \xi-\eta\vert} e^{\zeta(\eta)}[(2\bar U^2+4c\bar U\chi\circ y)y_\xi+h]d\eta\\ \nn 
    & \quad + \frac{e^{\zeta(\xi)}}{4}\int_\xi^\infty e^{-\vert \xi-\eta\vert}e^{-\zeta(\eta)}[(2\bar U^2+4c\bar U\chi\circ y)y_\xi+h]d\eta\\ \nn 
    & \quad + c^2g'\circ y\\ \nn
    & = \mathcal{Q}_1+\mathcal{Q}_2+c^2g'\circ y.
  \end{align*}
  Let $f(\xi)=\chi_{\xi>0}(\xi)e^{-\xi}$ and $A$ be the map defined by $A\colon v\mapsto f\star v$. Then $\mathcal{Q}_1$ can be rewritten as 
  \begin{equation}\label{eq:Q1}
    \mathcal{Q}_1(X)(\xi)=-\frac{e^{-\zeta(\xi)}}{4}A\circ R(\zeta, U, h)(\xi), 
  \end{equation}
  where $R$ is the operator from $E$ to $L^2(\Real)$ given by 
  \begin{equation}\label{eq:R}
    R(\zeta, U,h)=e^\zeta\Big( (2\bar U^2+4c\bar U\chi\circ y)y_\xi+h\Big).
  \end{equation}
  The Fourier transform of $f$ can be easily
  computed, and we obtain
  \begin{equation}\label{Fouf}
    \hat f(\eta)=\int_\Real f(\xi)e^{-2\pi\I\eta\xi}d\xi=\frac{1}{1+2\pi\I\eta}.
  \end{equation}
  The $H^1(\Real)$-norm can be expressed in terms of the Fourier transform as follows
  \begin{align*}
    \norm{f\star v}_{H^1}& =\norm{(1+\eta^2)^{\frac{1}{2}}\hat{f\star v}}_{L^2}\\ \nn 
    & = \norm{(1+\eta^2)^{\frac{1}{2}}\hat f \hat v}_{L^2}\\ \nn
    & \leq C\norm{\hat v}_{L^2}\\ \nn
    & =C\norm{v}_{L^2},
  \end{align*}
  for some constant $C$. Hence $A:L^2(\Real)\to
  H^1(\Real)$ is continuous.  Let us prove that
  $\mathcal{Q}_1$ is locally Lipschitz from $E$ to
  $H^1(\Real)$. It is not hard to prove that $R$
  is locally Lipschitz from $E$ to $L^2(\Real)$,
  by applying
  \begin{equation*}
    \vert \chi\circ y_1(\xi)-\chi\circ y_2(\xi)\vert =\vert \int_{y_1(\xi)}^{y_2(\xi)}\chi'(x)dx\vert \leq \norm{\chi'}_{L^\infty}\vert y_2(\xi)-y_1(\xi)\vert,
  \end{equation*}
  as the following estimate shows
  \begin{align}\nn
    & \norm{e^{\zeta_1} [(2\bar U^2_1 +4c_1\bar
      U_1\chi\circ
      y_1)y_{1,\xi}+h_1]-e^{\zeta_2}[(2\bar
      U^2_2+4c_2\bar U_2\chi\circ
      y_2)y_{2,\xi}+h_2]}_{L^2}\\ \nn &\leq
    \norm{(e^{\zeta_1}-e^{\zeta_2})[(2\bar
      U_1^2+4c_1\bar U_1\chi\circ
      y_1)y_{1,\xi}+h_1]}_{L^2}\\ \nn & \quad
    +\norm{e^{\zeta_2}[(2\bar U^2_1+4c_1\bar
      U_1\chi\circ y_1)y_{1,\xi}-(2\bar
      U^2_2+4c_2\bar U_2\chi\circ
      y_2)y_{2,\xi}+h_1-h_2]}_{L^2}\\
    \nn & \leq
    \norm{e^{\zeta_1}-e^{\zeta_2}}_{L^\infty}\norm{
      (2\bar U^2_1+4c_1\bar U_1\chi\circ
      y_1)y_{1,\xi}+h_1}_{L^2}\\ \nn & \quad
    +\norm{e^{\zeta_2}}_{L^\infty} \norm{(2\bar
      U^2_1+4c_1\bar U_1\chi\circ
      y_1)y_{1,\xi}-(2\bar U^2_2+4c_2\bar
      U_2\chi\circ y_2)y_{2,\xi}+h_1-h_2}_{L^2}\\
    \nn & \leq
    e^{\norm{\zeta_1}_{L^\infty}+\norm{\zeta_2}_{L^\infty}}\norm{\zeta_1-\zeta_2}_{L^\infty}\norm{(2\bar
      U_1^2+4c_1\bar U_1\chi\circ
      y_1)y_{1,\xi}+h_1}_{L^2}\\ \nn & \quad +
    e^{\norm{\zeta_2}_{L^\infty}}(\norm{h_1-h_2}_{L^2}+4\norm{c_1\bar
      U_1\chi\circ y_1-c_2\bar U_2\chi\circ
      y_2}_{L^2}+2\norm{\bar U_1^2\zeta_{1,\xi}-\bar
      U_2^2\zeta_{2,\xi}}_{L^2}\\ \nn & \quad
    \phantom{e^{\norm{\zeta_2}_{L^\infty}}}\qquad
    +4\norm{c_1\bar U_1\chi\circ
      y_1\zeta_{1,\xi}-c_2\bar U_2\chi\circ
      y_2\zeta_{2,\xi}}_{L^2}+2\norm{\bar
      U_1^2-\bar U_2^2}_{L^2})\\ \nn & \leq
    \norm{y_1-y_2}_{L^\infty}
    e^{\norm{\zeta_1}_{L^\infty}+\norm{\zeta_2}_{L^\infty}}\norm{(2\bar
      U_1^2+4c_1\bar U_1\chi\circ y_1)
      y_{1,\xi}+h_1}_{L^2}\\ \nn & \quad
    +e^{\norm{\zeta_2}_{L^\infty}}\Big(\norm{h_1-h_2}_{L^2}+2(\norm{\bar
      U_1}_{L^\infty}+\norm{\bar
      U_2}_{L^\infty})\norm{\bar U_1-\bar
      U_2}_{L^2}+4\vert c_1-c_2\vert \norm{\bar
      U_1}_{L^2}\\ \nn & \quad
    \phantom{e^{\norm{\zeta_2}_{L^\infty}}}\qquad
    +\vert c_2\vert \norm{\bar U_1-\bar
      U_2}_{L^2}+\vert c_2\vert \norm{\bar
      U_2}_{L^2}\norm{\chi\circ y_1-\chi\circ
      y_2}_{L^\infty}\\ \nn & \quad
    \phantom{e^{\norm{\zeta_2}_{L^\infty}}}\qquad
    +2\norm{\zeta_{1,\xi}}_{L^2}\norm{\bar
      U_1-\bar U_2}_{L^\infty}\norm{\bar U_1+\bar
      U_2}_{L^\infty}\\ \nn & \quad
    \phantom{e^{\norm{\zeta_2}_{L^\infty}}}\qquad
      +2\norm{\bar
      U_2}_{L^\infty}^2\norm{\zeta_{1,\xi}-\zeta_{2,\xi}}_{L^2}
 %     \\ \nn & \quad
 %   \phantom{e^{\norm{\zeta_2}_{L^\infty}}}\qquad
    +\norm{\zeta_{1,\xi}-\zeta_{2,\xi}}_{L^2}\norm{4c_1\bar
      U_1\chi\circ
      y_1}_{L^\infty}
    \\ \nn & \quad
    \phantom{e^{\norm{\zeta_2}_{L^\infty}}}\qquad          
      +4\norm{\zeta_{2,\xi}}_{L^2}\norm{c_1\bar
      U_1\chi\circ y_1-c_2\bar U_2\chi\circ
      y_2}_{L^\infty}\Big),
  \end{align}
  where
  \begin{align}\nn
    \norm{c_1\bar U_1\chi\circ y_1-c_2\bar
      U_2\chi\circ y_2}_{L^\infty} &\leq \vert
    c_1- c_2\vert\norm{\bar U_1}_{L^\infty}+\vert
    c_2\vert \norm{\bar U_1-\bar U_2}_{L^\infty}\\
    \nn & \quad+C\vert c_2\vert\norm{\bar U_2}_{L^\infty}
    \norm{y_1-y_2}_{L^\infty}.
  \end{align}
  Since $A$ is linear and continuous from $L^2(\Real)$ to $H^1(\Real)$, the composition $A\circ R$ is locally Lipschitz from $E$ to $H^1(\Real)$.
  Then, we use the following lemma, which is stated without proof.
  \begin{lemma}\label{lem:BLip}
    Let $\mathcal{R}_1\colon E\to V$ and $\mathcal{R}_2:E\to H^1(\Real)$, or $\mathcal{R}_2:E\to V$ be two locally Lipschitz maps. Then the product $X\to \mathcal{R}_1(X)\mathcal{R}_2(X)$ is also locally Lipschitz from $E$ to $H^1(\Real)$, or from $E$ to $V$. 
  \end{lemma}

  Since the mapping $X\mapsto e^{-\zeta}$ is
  locally Lipschitz from $E$ to $V$, the function
  $\mathcal{Q}_1$ is the product of two locally
  Lipschitz maps, one from $E$ to $H^1(\Real)$ and
  the other one from $E$ to $V$, it is locally
  Lipschitz from $E$ to $H^1(\Real)$. Similarly
  one proves that $\mathcal{Q}_2$ is locally
  Lipschitz.  
  Thus it is left to show that $X\mapsto g'\circ y$ is locally Lipschitz from $E$ to $H^1(\Real)$. 
  By \eqref{eq:gprime} we have 
  \begin{align*}
 g'(y(\xi))&=\int_{-\infty}^\xi e^{-(y(\xi)-y(z))}\chi'(y(z))\chi(y(z))y_\xi(z)dz\\
&\quad +\int_\xi^\infty e^{-(y(z)-y(\xi))}\chi'(y(z))\chi(y(z))y_\xi(z)dz\\
&= I_1(\xi)+I_2(\xi).
\end{align*} 
Introduce $v(z)=e^{\zeta(z)}\chi^\prime
(y(z))\chi(y(z))y_\xi(z)$, then we can write
$I_1(\xi)$ as
\begin{equation}
  I_1(\xi)=e^{-\zeta(\xi)}A(v)
\end{equation}
and we only have to check that the mapping
$X\mapsto v$ is locally Lipschitz from $E$ to
$L^2(\Real)$. This follows from the smoothness of $\chi$
and the fact that
\begin{align*}
  \norm{\chi^{\prime}(y(\xi))}_{L^2}&\leq\norm{\chi^\prime}_{L^\infty}(\meas\{\xi\in\Real\
  |\ y(\xi)\in[0,1]\})^{1/2}\\
  &\leq\norm{\chi^\prime}_{L^\infty}(\meas\{[-\norm{\zeta}_{L^\infty},1+\norm{\zeta}_{L^\infty}]\})^{1/2}\\
  &\leq C.
\end{align*}
Therefore $\mathcal{Q}$ is locally Lipschitz from
$E$ to $H^1(\Real)$.  To prove that
$\mathcal{P}-U^2$ is locally Lipschitz from $E$ to
$H^1(\Real)$ one can use the same techniques after
discovering that one can write, using
\eqref{eq:defG},
\begin{align*}
  P(\xi)-U(\xi)^2& =\frac{1}{4}\int_\Real
  e^{-\vert y(\xi)-y(\eta)\vert}\Big( (2\bar
  U^2+4c\bar U\chi(y))y_\xi+h\Big)(\eta)d\eta \\
  \nn &\quad +\frac{1}{2}\int_\Real e^{-\vert
    y(\xi)- y(\eta)\vert}
  \Big(2c^2(\chi'(y))^2+2c^2\chi(y)\chi''(y)\Big)y_\xi(\eta)
  d\eta\\ \nn &\quad -\bar U(\xi)^2-2\bar
  U(\xi)c\chi(y(\xi)).
  \end{align*}
  Hence $\mathcal{Q}$ and $\mathcal{P}-U^2$ are
  locally Lipschitz continuous from $E$ to
  $H^1(\Real)$.
\end{proof}

By the above lemma we have that $Q\in H^1(\Real)$ and therefore $\lim_{\xi\to\pm\infty}Q(t,\xi)=0$. Hence \eqref{eq:govsys} implies that if a solution in $E$ exists the asymptotic behavior of $U(t,\xi)$ must be preserved for all times. Thus using that $y-\id\in L^\infty(\Real)$ we can write $U(t,\xi)=\bar U(t,\xi)+c\chi\circ y(t,\xi)$. Therefore,
we can also write \eqref{eq:govsys} as
\begin{subequations}
  \label{eq:govsyss}
  \begin{align}
    \label{eq:govsyss1}
    y_t&=U,\\
    \label{eq:govsyss2}
    \bar U_t&=-Q-c(\chi'\circ y)U,\\
    \label{eq:govsyss3}
    c_t&=0,\\
    \label{eq:govsyss4}
    h_t&=2(U^2-P)U_\xi.
  \end{align}
\end{subequations}

\begin{theorem}\label{thm:short}
 Given $X_0=(\zeta_0, U_0, h_0)$ in $E$, then there exists a time $T$ depending only on $\norm{X_0}_E$ such that \eqref{eq:govsys} admits a unique solution in $C^1([0,T],E)$ with initial data $X_0$.
\end{theorem}

\begin{proof}  
 Solutions of \eqref{eq:govsys} can be rewritten as 
\begin{equation*}
 X(t)=X_0+\int_0^t F(X(\tau))d\tau,
\end{equation*}
where $F\colon E\to E$ is defined by the
right-hand side of \eqref{eq:govsys}. The
integrals are defined as Riemann integrals of
continuous functions on the Banach space
$E$. Using Lemma~\ref{lem:PQ} we can check that
$F(X)$ is a Lipschitz function on bounded sets of
$E$. Since $E$ is a Banach space, we use the
standard contraction argument to prove the
theorem.
\end{proof}

After differentiating \eqref{eq:govsys} we obtain
\begin{subequations}
 \label{eq:govsysder}
\begin{align}
 \label{eq:govsysder1}
 y_{\xi,t}& =U_\xi,\\ 
\label{eq:govsysder2}
\bar U_{\xi,t} &= \frac{1}{2}h+(U^2-P)y_\xi-c\chi''\circ yy_\xi U+c\chi'\circ yQ,\\
\label{eq:govsysder3}
U_{\xi,t}&=\frac{1}{2}h+(U^2-P)y_\xi, \\
\label{eq:govsysder4}
h_t&=2(U^2-P)U_\xi. 
\end{align}
\end{subequations}
We define the set $\G$ as follows.
\begin{definition}
 The set $\mathcal{G}$ is composed of all $(\zeta, U, h)\in E$ such that 
\begin{subequations}\label{eq:G}
 \begin{align}\label{eq:G1}
   &(\zeta, U)\in [W^{1,\infty}(\Real)]^2,\quad
   h\in L^{\infty}(\Real), \\ \label{eq:G2} 
   & y_\xi\geq 0, \quad h\geq 0, \quad y_\xi+h>0 \text{ almost
     everywhere},\\ \label{eq:G3} 
     & y_\xi h=U_\xi^2 \text{ almost everywhere},
 \end{align}
\end{subequations}
where we denote $y(\xi)=\xi+\zeta(\xi)$.
\end{definition}
As in \cite[Lemma 2.7]{HR}, we can prove that the
set $\G$ is preserved by the flow, that is, for
any initial data $X_0=(\zeta_0, U_0, h_0)$ in
$\mathcal{G}$, if $X(t)=(\zeta(t), U(t), h(t))$ is
the short-time solution of \eqref{eq:govsys} in
$C^1([0,T],E)$ for some $T>0$ with initial data
$(\zeta_0, U_0, h_0)$, then $X(t)$ belongs to
$\mathcal{G}$ for all $t\in[0,T]$. Moreover we
have that, for almost every $t\in[0,T]$,
\begin{equation}
  \label{eq:posyxi}
  y_\xi(t,\xi)>0\text{ for almost every }\xi\in\Real.
\end{equation}
Using this property, we can derive the necessary
estimate to prove the global existence of
solutions to \eqref{eq:govsys}.
%---------- thm -------------
\begin{theorem}\label{th:global}
  For any $X_0=(y_0, U_0, h_0)\in\mathcal{G}$, the
  system \eqref{eq:govsys} admits a unique global
  solution $X(t)=(y(t), U(t), h(t))$ in
  $C^1([0,\infty), E)$ with initial data
  $X_0=(y_0,U_0, h_0)$. We have
  $X(t)\in\mathcal{G}$ for all times. If we equip
  $\mathcal{G}$ with the topology inducted by the
  $E$-norm, then the mapping $S:\mathcal{G}\times
  [0,\infty)\to \mathcal{G}$ defined as
\begin{equation}\nn
S_t(X_0)=X(t)
\end{equation}
is a continuous semigroup.
\end{theorem}
\begin{proof} %See \cite[Theorem 2.8]{HR}.
  The solution has a finite time of existence $T$
  only if $\norm{(\zeta,U,h)(t,\dott)}_E$ blows up
  when $t$ tends to $T$ because, otherwise, by
  Theorem \ref{thm:short}, the solution can be
  prolongated by a small time interval beyond
  $T$. Let $(\zeta,U,h)$ be a solution of
  \eqref{eq:govsys} in $C^1([0,T),E)$ with initial
  data $(\zeta_0,U_0,h_0)$. We want to prove that
  \begin{equation}
    \label{eq:enormbound}
    \sup_{t\in[0,T)}\norm{(\zeta(t,\dott), U(t,\dott),h(t,\dott))}_E<\infty.
  \end{equation}
  We can follow the proof of \cite[Theorem
  2.8]{HR} once we have established that
  \begin{equation}
    \label{eq:inftybdPQ}
    \sup_{t\in[0,T)}\left(\norm{U(t,\cdot)}_{L^\infty}+\norm{P(t,\cdot)}_{L^\infty}+\norm{Q(t,\cdot)}_{L^\infty}\right)<\infty.
  \end{equation}
  Let us introduce
 \begin{equation*}
    \Gamma=\int_{\Real} \bar U^2y_\xi\,d\xi+\norm{h}_{L^1}.
  \end{equation*}
  By \eqref{eq:G3}, we have
  \begin{equation}\label{eq:hL1}
    h=U_\xi^2-\zeta_\xi h
  \end{equation}
  and therefore $h\in L^1(\Real)$. Moreover since $h\geq
  0$, we have $\norm{h}_{L^1}=\int_\Real
  h\,d\xi$. 
  We can estimate the $\norm{\bar U}_{L^\infty}^2$
  as follows.
  \begin{subequations}
    \label{eq:derivULinf}
    \begin{align*}
      \bar U^2(\xi)&=2\int_{-\infty}^{\xi}\bar U\bar U_\xi\,d\eta\\
      &=2\int_{-\infty}^{\xi}\bar
      UU_\xi\,d\eta-2\int_{-\infty}^{\xi}c\bar
      U\chi'\circ y
      y_\xi\,d\eta\\
      &\leq \int_{\{\xi\mid  y_\xi(\xi)>0\}}\bar
      U^2y_\xi+\frac{U_\xi^2}{y_\xi}\,d\eta-2\int_{\Real}c\bar
      U\chi'\circ y
      y_\xi\,d\eta\\
      &\leq \int_{\{\xi\mid  y_\xi(\xi)>0\}}(\bar{U}^2y_\xi+h)\,d\eta+2C\norm{\bar U}_{L^\infty}\\
      &\leq \Gamma+2C\norm{\bar U}_{L^\infty}.
    \end{align*}
  \end{subequations}
  After using that $\norm{\bar U}_{L^\infty}\leq
  \frac{1}{4C}\norm{\bar U}_{L^\infty}^2+C$, we
  get
  \begin{equation}\label{est:Uinf}
    \norm{\bar U}_{L^\infty}^2\leq 2\Gamma+C.
  \end{equation}
  From \eqref{eq:defP}, we get
  \begin{align}\label{est:Pinf}
    \norm{P}_{L^\infty}&\leq \frac12(\norm{\bar
      U}_{L^\infty}^2+2\vert c\vert \norm{\bar
      U}_{L^\infty})\int_\Real
    e^{-\abs{y(\xi)-y(\eta)}}y_\xi\,d\eta+\Gamma+c^2\\
    \notag
    &\leq (2\norm{\bar
      U}_{L^\infty}^2+\vert c\vert^2)+\Gamma+\vert c\vert^2 \\
    \notag
    &\leq 5\Gamma+C.
\end{align}
Similarly, one obtains that
\begin{equation}
  \label{est:Qinf}
  \norm{Q}_{L^\infty}\leq 5\Gamma+C.
\end{equation}
Hence, \eqref{eq:inftybdPQ} will be proved when
we prove that $\sup_{t\in[0,T)}\Gamma<\infty$. We
can now compute the variation of $\Gamma$. From
\eqref{eq:govsyss} we get
\begin{align*}
  \frac{d\Gamma}{dt}&=\int_{\Real}2\bar U\bar
  U_ty_\xi\,d\xi+\int_{\Real}\bar U^2y_{\xi
    t}\,d\xi+\int_\Real h_t\,d\xi\\
  &=\int_\Real2\bar U(-Q-cU\chi'\circ y)y_\xi\,d\xi+\int_\Real \bar
  U^2U_\xi\,d\xi+\int_{\Real}2(U^2-P)U_\xi\,d\xi.
\end{align*}
We estimate each of these three integrals, that we
denote $A_1$, $A_2$ and $A_3$, respectively.
We have 
\begin{align*}
 A_1& = -2\int_\Real Q\bar Uy_\xi d\xi-2\int_\Real c\bar U^2\chi'\circ yy_\xi+c^2\bar U\chi\circ y\chi'\circ yy_\xi d\xi\\ 
& \leq -2\int_\Real P_\xi\bar Ud\xi+C\Gamma+C\norm{\bar U}_{L^\infty}\\
& \leq 2\int_\Real P\bar U_\xi d\xi+C\Gamma+C\norm{\bar U}_{L^\infty},
\end{align*}
%\begin{align*}
 % A_1&\leq\int_{\Real}-Q\bar
  %Uy_\xi\,d\xi+C\Gamma+C\norm{\bar U}_{L^\infty}\\
  %&=\int_{\Real}P\bar
  %U_\xi\,d\xi+C\Gamma+C\norm{\bar U}_{L^\infty},
%\end{align*}
after integration by parts, since $P_\xi=Qy_\xi$.
We have
\begin{align*}
  A_2&=\int_\Real \bar U^2\bar
  U_\xi\,d\xi+\int_\Real c\bar U^2(\chi'\circ
  y)y_\xi d\xi\\
  &=\int_\Real c\bar U^2(\chi'\circ y)y_\xi d\xi\leq
  C\Gamma.
\end{align*}
We have
\begin{align*}
 A_3& = 2\int_\Real U^2 U_\xi d\xi-2\int_\Real P\bar U_\xi d\xi-2c\int_\Real P\chi'\circ yy_\xi d\xi\\
& = -2\int_\Real P\bar U_\xi d\xi -2c\int_\Real P\chi'\circ yy_\xi d\xi+\frac{2}{3}c^3\\ 
& \leq -2\int_\Real P\bar U_\xi d\xi +C\norm{P}_{L^\infty}+C.
\end{align*}

%\begin{align*}
 % A_3&=\int_{\Real}2(\bar U^2\bar U_\xi+c\bar U^2\chi'\circ yy_\xi)\,d\xi+4c\int_\Real\bar U\chi\circ y U_\xi\,d\xi\\
 % &\quad +2c^2\int_\Real((\chi\circ y)^2(\bar
 % U_\xi+c\chi'\circ yy_\xi))\,d\xi-\int_\Real
 % P(\bar U_\xi+c\chi'\circ y y_\xi)\,d\xi
%\end{align*}
%We denote by $A_{31}$, $A_{32}$, $A_{33}$ and
%$A_{34}$ these four integrals. We have
%\begin{align*}
%  A_{31}=2c\int_{\Real}\bar U^2\chi\circ
%  yy_\xi\,d\xi\leq C\Gamma.
%\end{align*}
%We have
%\begin{align*}
%  A_{32}\leq C\int_{\Real}\abs{\bar
%    U}\abs{U_\xi}\,d\xi\leq \frac{C}2\Gamma,
%\end{align*}
%see derivation of estimate \eqref{eq:derivULinf}. We have
%\begin{align*}
%  A_{33}&=-2c^2\int_\Real(2(\chi\circ y)(\chi'\circ
%  y)y_\xi\bar U)\,d\xi+2c^3\int_\Real((\chi\circ y)^2\chi'\circ
%  yy_\xi)\,d\xi\\
%  &\leq C\norm{U}_{L^\infty}+C.
%\end{align*}
%We have
%\begin{align*}
%  A_{34}&\leq -\int_\Real P\bar U_\xi\,d\xi+C\norm{P}_{L^\infty}\\
%  &\leq -\int_\Real P\bar U_\xi\,d\xi+3C\Gamma+C,
%\end{align*}
%by \eqref{est:Pinf}. 
Finally, by adding up all
these estimates, we get
\begin{align*}
 \frac{d\Gamma}{dt}& \leq C\Gamma+C\norm{\bar U}_{L^\infty}+C\norm{P}_{L^\infty}\\
& \leq C\Gamma +C+C\norm{\bar U}_{L^\infty}^2+C\norm{P}_{L^\infty}\\
& \leq C\Gamma+C,
\end{align*}
%\begin{align*}
 % \frac{d\Gamma}{dt}&\leq
 % C\Gamma+C\norm{U}_{L^\infty}+C\\
 % &\leq C\Gamma+C(\norm{U}_{L^\infty}^2+1)+C\\
 % &\leq 3C\Gamma+C,
%\end{align*}
by \eqref{est:Uinf} and \eqref{est:Pinf}.  Hence,
Gronwall's lemma implies that
$\sup_{t\in[0,T)}\Gamma(t)<\infty$ Using now
\eqref{est:Uinf}, \eqref{est:Pinf}, and
\eqref{est:Qinf}, we immediately obtain that the
same is true for $\norm{\bar
  U(t,\dott)}_{L^\infty}$,
$\norm{P(t,\dott)}_{L^\infty}$, and
$\norm{Q(t,\dott)}_{L^\infty}$, which are bounded
by a constant only dependent on
$\sup_{t\in[0,T)}\Gamma(t)<\infty$ for $t\in
[0,T)$.
\end{proof}

\section{From Eulerian to Lagrangian coordinates and vice versa}
The appropriate set to construct a semigroup of
conservative solutions is the set $\D$ defined
below, which allows for concentration of the
energy in domains of zero measure.
\begin{definition}\label{def:euler}
  The set $\mathcal{D}$ is composed of all pairs
  $(u,\mu)$ such that $u\in H_{0,\infty}(\Real)$ and
  $\mu$ is a positive finite Radon measure whose
  absolutely continuous part, $\mu_{ac}$,
  satisfies
\begin{equation}
  \mu_{ac}=u_x^2dx.
\end{equation}
\end{definition}
The system \eqref{eq:govsys} is invariant with
respect to relabeling. Relabeling is modeled by
the action of the group of diffeomorphisms $G$
that we now define.
\begin{definition}\label{def:G}
  We denote by $G$ the subgroup of the group of
  homeomorphisms from $\Real$ to $\Real$ such that
\begin{subequations}
\label{eq:Gcond}
 \begin{align}
  \label{eq:Gcond1}
  f-\id \text{ and } f^{-1}-\id &\text{ both belong to } W^{1,\infty}(\Real), \\
  \label{eq:Gcond2}
  f_\xi-1 &\text{ belongs to } L^2(\Real),
 \end{align}
\end{subequations}
where $\id$ denotes the identity function. Given $\alpha>0$, we denote by $G_\alpha$ the subset $G_\alpha$ of $G$ defined by 
\begin{equation}
 G_\alpha=\{ f\in G\mid  \norm{f-\id}_{W^{1,\infty}}+\norm{f^{-1}-\id}_{W^{1,\infty}}\leq\alpha\}. 
\end{equation}
\end{definition}
We define the subsets $\mathcal{F}_\alpha$ and
$\mathcal{F}$ of $\mathcal{G}$ as follows
\begin{equation*}
\mathcal{F}_\alpha=\{X=(y,U,h)\in\mathcal{G}\mid  y+H\in G_\alpha\},
\end{equation*}
and
\begin{equation*}
\mathcal{F}=\{X=(y,U,h)\in\mathcal{G}\mid  y+H\in G\},
\end{equation*}
where $H(t,\xi)$ is defined by 
\begin{equation*}
 H(t,\xi)=\int_{-\infty}^\xi h(t,\tau)d\tau,
\end{equation*}
which is finite since, from \eqref{eq:G3}, we have
$h=U_\xi^2-\zeta_\xi h$ and therefore $h\in
L^1(\Real)$.  For $\alpha=0$, we have
$G_0=\{\id\}$. As we shall see, the space
$\mathcal{F}_0$ will play a special role. These
sets are relevant only because they are preserved
by the governing equation \eqref{eq:govsys}. In
particular, while the mapping $\xi\mapsto
y(t,\xi)$ may not be a diffeomorphism for some
time $t$, the mapping $\xi\mapsto
y(t,\xi)+H(t,\xi)$ remains a diffeomorphism for
all times $t$. As in \cite[Lemma 3.3]{HR}, we can
establish that the space $\mathcal{F}$ is
preserved by the governing equations
\eqref{eq:govsys}. More precisely, we have that,
given $\alpha$, $T\geq 0$, and $X_0\in
\mathcal{F}_\alpha$,
\begin{equation*}
 S_t(X_0)\in \mathcal{F}_{\alpha'}, 
\end{equation*}
for all $t\in[0,T]$ where $\alpha'$ only depends
on $T$, $\alpha$ and $\norm{X_0}_E$. For the sake
of simplicity, for any $X=(y, U, h)\in
\mathcal{F}$ and any function $f\in\Gr$, we denote
$(y\circ f, U\circ f, h\circ ff_\xi)$ by $X\circ
f$. Then, $X\circ f$ corresponds to the
relabeling of $X$ with respect to the relabeling
function $f\in G$. The map from $\Gr\times\F$ to
$\F$ given by $(f,X)\mapsto X\circ f$ defines an
action of the group $\Gr$ on $\F$, see
\cite[Proposition 3.4]{HR}. Since $\Gr$ is acting
on $\F$, we can consider the quotient space
$\quot$ of $\F$ with respect to the action of the
group $G$. The equivalence relation on $\F$ is
defined as follows: For any $X,X'\in\F$, we say
that $X$ and $X'$ are equivalent if there exists
$f\in\Gr$ such that $X'=X\circ f$. We denote by
$\Pi(X)=[X]$ the projection of $\F$ into the
quotient space $\quot$, and introduce the mapping
$\Gamma\colon\F\rightarrow\F_0$ given by
\begin{equation*}
\Gamma(X)=X\circ (y+H)\inv
\end{equation*}
for any $X=(y,U,h)\in\F$. We have $\Gamma(X)=X$
when $X\in\F_0$ and $\Gamma$ is invariant under
the $\Gr$ action, that is, $\Gamma(X\circ
f)=\Gamma(X)$ for any $X\in\F$ and
$f\in\Gr$. Hence, we can define a mapping
$\tilde\Gamma$ from the quotient space $\quot$ to
$\F_0$ as $\tilde\Gamma([X])=\Gamma(X)$ and the
sets $\F_0$ and $\F$ are in bijection as
\begin{equation*}
  \tilde\Gamma\circ\Pi|_{\F_0}=\id|_{\F_0}.
\end{equation*}
We equip $\F_0$ with the metric induced by the
$E$-norm, i.e., $d_{\F_0}(X,X')=\norm{X-X'}_E$ for
all $X,X'\in\F_0$. Since $\F_0$ is closed in $E$,
this metric is complete. We define the metric on
$\quot$ as
\begin{equation*}
d_\quot([X],[X'])=\norm{\Gamma(X)-\Gamma(X')}_E,
\end{equation*}
for any $[X],[X']\in\quot$. Then, $\quot$ is
isometrically isomorphic with $\F_0$ and the
metric $d_\quot$ is complete.

We denote by $S\colon\F\times
[0,\infty)\rightarrow \F$ the continuous semigroup
which to any initial data $X_0\in \F$ associates
the solution $X(t)$ of the system of differential
equations \eqref{eq:govsys} at time $t$. As
indicated earlier, the Camassa--Holm equation is
invariant with respect to relabeling.  More
precisely, using our terminology, we obtain the
diagram 
\begin{equation}
\label{eq:diag}
\xymatrix{
\F_0\ar[r]^{\Pi}&\quot\\
\F_\alpha\ar[u]^{\Gamma}&\\
\F_0\ar[u]^{S_t}\ar[r]^\Pi&\quot\ar[uu]_{\tilde S_t}
}
\end{equation}
which  summarizes the following theorem.
\begin{theorem}
\label{th:sgS} 
For any $t>0$, the mapping $S_t\colon\F\rightarrow\F$
is $\Gr$-equivariant, that is,
\begin{equation}
\label{eq:Hequivar}
S_t(X\circ f)=S_t(X)\circ f
\end{equation}
for any $X\in\F$ and $f\in\Gr$. Hence, the mapping
$\tilde S_t$ from $\quot$ to $\quot$ given by
\begin{equation*}
\tilde S_t([X])=[S_tX]
\end{equation*}
is well-defined. It generates a continuous
semigroup.
\end{theorem}
\begin{proof} See \cite[Theorem 3.7]{HR}.
\end{proof}
Note that the continuity of $\tilde S_t$ holds
because, for any given $\alpha\geq 0$, the
restriction of $\Gamma$ to $\F_\alpha$ is a
continuous mapping from $\F_\alpha$ to $\F_0$, see
\cite[Lemma 3.5]{HR}.

Our next task is to derive the correspondence
between Eulerian coordinates (functions in $\D$)
and Lagrangian coordinates (functions in
$\quot$). Let us denote by
$L\colon\D\rightarrow\quot$ the mapping
transforming Eulerian coordinates into Lagrangian
coordinates whose definition is contained in the
following theorem.
\begin{definition}
\label{th:Ldef}
For any $(u,\mu)$ in $\D$, let
\begin{subequations}
\label{eq:Ldef}
\begin{align}
\label{eq:Ldef1}
y(\xi)&=\sup\left\{y\mid  \mu((-\infty,y))+y<\xi\right\},\\
\label{eq:Ldef3}
U(\xi)&=u\circ y(\xi),\\
\label{eq:Ldef2}
h(\xi)&=1-y_\xi(\xi).
\end{align}
\end{subequations}
Then $(y,U,h)\in\F_0$. We define
$L(u,\mu)=(y,U,h)$.
\end{definition}
Note that the mapping $L$ depends on the partition
function $\chi$. The well-posedness of this
definition is established in the same way as in
\cite[Theorem 3.8]{HR}. In the other direction, we
obtain $\mu$, the energy density in Eulerian
coordinates, by pushing forward by $y$ the energy
density in Lagrangian coordinates $h d\xi$. We are
led to the mapping $M$ which transforms Lagrangian
coordinates into Eulerian coordinates and whose
definition is contained in the following theorem.
\begin{definition}
\label{th:umudef} 
Given any element $X$ in $\F/\G$. Then $(u,\mu)$
defined as follows
\begin{subequations}
\label{eq:umudef}
\begin{align}
\label{eq:umudef1}
&u(x)=U(\xi)\text{ for any }\xi\text{ such that  }  x=y(\xi),\\
\label{eq:umudef2}
&\mu=y_\#(h\,d\xi)
\end{align}
\end{subequations}
belongs to $\D$. We denote by
$M\colon\F/\G\rightarrow\D$ the mapping which to
any $X\in\F/\G$ associates $(u,\mu)$ as given by
\eqref{eq:umudef}.
\end{definition}

The well-posedness of this definition is
established in the same way as in \cite[Theorem
3.11]{HR}. Finally, one can show that the
transformation from Eulerian to Lagrangian
coordinates is a bijection (see \cite[Theorem
3.12]{HR}).
%------------- theorem
\begin{theorem}\label{th:LMinv}The mappings $M$ and
  $L$ are invertible. We have
\begin{equation*}
L\circ M=\id_\quot\text{ and }M\circ L=\id_\D.
\end{equation*}
\end{theorem}

\section{Continuous semigroup of solutions on $\D$}

For each $t\in\Real$, we define the mapping $T_t$
from $\D$ to $\D$ as
\begin{equation}\label{eq:Tt}
T_t=M\tilde S_tL.
\end{equation}
It corresponds to the following diagram:
\begin{equation*}
\xymatrix{ \D&\quot\ar[l]_M\\
\D\ar[u]^{T_t}\ar[r]^L&\quot\ar[u]_{\tilde S_t}}
\end{equation*}

We define global weak conservative solution to the
Camassa-Holm equation as follows.
\begin{definition}
  \label{eq:defweakconssol}
  Assume that $u\colon[0,\infty)\times\Real \to \Real$  satisfies \\
  (i) $u\in L^\infty_{\rm{loc}}([0,\infty), H_\infty(\Real))$, \\
  (ii) the equations
  \begin{multline}\label{eq:weak1}
    \iint_{[0,\infty)\times\Real}\Big[
    -u(t,x)\phi_t(t,x)
    +\big(u(t,x)u_x(t,x)+P_x(t,x)\big)\phi(t,x)\Big]dxdt\\
    =\int_\Real u(0,x)\phi(x,0)dx
  \end{multline}
  and
  \begin{equation}\label{eq:weak2}
    \iint_{[0,\infty)\times\Real} \Big[(P(t,x)-u^2(t,x)-\frac{1}{2}u_x^2(t,x))\phi(t,x)+P_x(t,x)\phi_x(t,x)\Big]dxdt=0,
  \end{equation}
  hold for all $\phi\in
  C^\infty_0([0,\infty)\times\Real)$. Then we say that
  $u$ is a weak global solution of the
  Camassa--Holm equation. If $u$ in addition satisfies
  \begin{equation*}
    (u^2+u_x^2)_t+(u(u^2+u_x^2))_x-(u^3-2Pu)_x=0
  \end{equation*}
  in the sense that
  \begin{multline}\label{eq:weak3}
    \iint_{(0,\infty)\times\Real}\Big[
    (u^2(t,x)+u_x^2(t,x))\phi_t(t,x)+(u(t,x)(u^2(t,x)+u_x^2(t,x)))\phi_x(t,x)\\-(u^3(t,x)-2P(t,x)u(t,x))\phi_x(t,x)\Big]dxdt
    =0,
  \end{multline}
  for any $\phi\in
  C_0^{\infty}((0,\infty)\times\Real)$, we say that $u$
  is a weak global conservative solution of the
  Camassa--Holm equation.
\end{definition}

On $\D$ we define the
distance $d_\D$ which makes the bijection $L$
between $\D$ and $\quot$ into an isometry:
\begin{equation*}
  d_\D((u,\mu),(\bar u,\bar \mu))=d_\quot(L(u,\mu),L(\bar u,\bar \mu)).
\end{equation*}
Since $\quot$ equipped with $d_\quot$ is a
complete metric space, the space $\D$ equipped
with the metric $d_D$ is a complete metric
space. Our main theorem then reads as follows.
%-------- main thm
\begin{theorem}
 \label{th:main}
 The semigroup $(T_t, d_\D)$ is a continuous
 semigroup on $\D$ with respect to the metric
 $d_\D$. Moreover, given any intial condition
 $(u_0,\mu_0)\in\D$, we denote
 $(u,\mu)(t)=T_t(u_0,\mu_0)$. Then $u(t,x)$ is a
 weak global conservative solution of the
 Camassa--Holm equation. Moreover, letting
 $\nu=u^2\,dx+\mu$, we have
 \begin{equation*}
   \nu_t+(u\nu)_x-(u^3-2Pu)_x=0
 \end{equation*}
 in the sense of distribution, that is,
 \begin{multline}\label{eq:weak4}
   \iint_{[0,\infty)\times\Real}(\phi_t(t,x)+u(t,x)\phi_x(t,x))d\nu(t,x)dt\\
   -\iint_{[0,\infty)\times\Real}(u^3(t,x)-2P(t,x)u(t,x))\phi_x(t,x)dxdt
   =-\int_\Real \phi(0,x)d\nu(0,x)dx.
 \end{multline}

\end{theorem}
%--------- proof of main thm
\begin{proof}
First we prove that $T_t$ is a semigroup. Since $\tilde S_t$ is a mapping from $\F_0$ to $\F_0$, we have  
\begin{equation*}
 T_tT_{t'}=M\tilde S_tLM\tilde S_{t'}L=M\tilde S_t\tilde S_{t'}L=M\tilde S_{t+t'}L=T_{t+t'}, 
\end{equation*}
where we used \eqref{eq:Tt} and the semigroup
property of $\tilde S_t$.  To show that $u(t,x)$
is a weak global solutions, we have to show that
\eqref{eq:weak1} and \eqref{eq:weak2} are
satisfied. The proof of \eqref{eq:weak1} and \eqref{eq:weak2} is
essentially the same as in \cite{HR}.  Let us
check that \eqref{eq:weak4} is fullfilled.  After
making the change of variable $x=y(t,\xi)$ we
obtain
\begin{align*}
 \iint_{[0,\infty)\times\Real}& u^2(t,x)\phi_t(t,x)dxdt
= \iint_{[0,\infty)\times\Real} U^2(t,\xi)\phi_t(t,y(t,\xi))y_\xi(t,\xi)d\xi dt\\
& = \iint_{[0,\infty)\times\Real} U^2(t,\xi)\Big( (\phi(t,y(t,\xi)))_t-\phi_x(t,y(t,\xi))y_t(t,\xi)\Big)y_\xi(t,\xi)d\xi dt\\
& = -\int_\Real U^2(0,\xi)\phi(0,y(0,\xi))d\xi -\iint_{[0,\infty)\times\Real} U^3(t,\xi)\phi_\xi(t,y(t,\xi))d\xi dt\\
& \quad +\iint_{[0,\infty)\times\Real} \Big(2U(t,\xi)Q(t,\xi)y_\xi(t,\xi)-U^2(t,\xi)U_\xi(t,\xi)\Big)\phi(t, y(t,\xi))d\xi dt\\
& =- \int_\Real u^2(0,x)\phi(0,x)dx\\
&  \quad +\iint_{[0,\infty)\times\Real}\Big(2u(t,x)P_x(t,x)-u^2(t,x)u_x(t,x)\Big)\phi(t,x)-u^3(t,x)\phi_x(t,x)dxdt,
\end{align*}
and 
\begin{align*}
 \iint_{[0,\infty)\times\Real}& \phi_t(t,x)d\mu(t,x) dt
=\iint_{[0,\infty)\times\Real} \phi_t(t,y(t,\xi))h(t,\xi)d\xi dt\\
& = \iint_{[0,\infty)\times\Real} \Big( (\phi(t,y(t,\xi)))_t-\phi_x(t,y(t,\xi))y_t(t,\xi)\Big)h(t,\xi)d\xi dt\\
& = -\int_{\Real}\phi(0,y(0,\xi))h(0,\xi)d\xi \\
&\quad -\iint_{[0,\infty)\times\Real} h_t(t,\xi)\phi(t,y(t,\xi))+U(t,\xi)h(t,\xi)\phi_x(t,y(t,\xi))d\xi dt\\
& =- \int_{\Real}\phi(0,y(0,\xi))h(0,\xi)d\xi \\
&\quad -\iint_{[0,\infty)\times\Real} 2(U^2(t,\xi)-P(t,\xi))U_\xi(t,\xi)\phi(t,y(t,\xi))+U(t,\xi)h(t,\xi)\phi_x(t,y(t,\xi))d\xi dt\\
& = -\int_{\Real}\phi(0,x)d\mu(0,x)dx-\iint_{[0,\infty)\times\Real}\phi_x(t,x)u(t,x)d\mu(t,x) dt\\
& \quad  -\iint_{[0,\infty)\times\Real} 2(u^2(t,x)-P(t,x))u_x(t,x)\phi(t,x) dxdt.
\end{align*}
This finishes the proof. Since for almost every
$t\in[0,T]$, $y_\xi(t,\xi)>0$ for almost every
$\xi\in\Real$, see \eqref{eq:posyxi}, the property
\eqref{eq:weak3} follows from \eqref{eq:weak4}.
\end{proof} 

\section{Invariance of the topology with respect
  to the choice of the partition function}\label{top1}

The mappings $L$ and $M$ depend on the choice of
the partition function $\chi$. To emphasize this
dependence, we write $L_{\chi}$ and $M_{\chi}$. In
this section we prove that different choices of
the partition function $\chi$ lead to the same
topology in $\D$. Given two partition functions
$\chi$ and $\tilde\chi$, we obtain two topologies
\begin{equation}
  \label{eq:defdD}
  d_\D((u,\mu),(\bar u,\bar \mu))=\norm{L(u,\mu)-L(\bar u,\bar \mu)}_{E_{\chi}},
\end{equation}
and
\begin{equation}
  \label{eq:defdDt}
  \tilde d_\D((u,\mu),(\bar u,\bar \mu))=\norm{L(u,\mu)-L(\bar u,\bar \mu)}_{E_{\tilde\chi}}.
\end{equation}
In \eqref{eq:defdD} and \eqref{eq:defdDt}, we add
the subscripts $\chi$ and $\tilde\chi$ to indicate
which norm is used on $E$.

\begin{theorem}
  \label{th:sametop}
  We consider two partition functions $\chi$ and
  $\tilde\chi$. Then, the metric they induce on
  $\D$ are equivalent, that is, there exists a
  constant $C>0$ which only depends on $\chi$ and
  $\tilde\chi$ such that
  \begin{equation*}
    \frac1{C} 
    \tilde d_D((\bar u,\bar\mu),(u,\mu))\leq d_D((\bar u,\bar \mu),(u,\mu))\leq C \tilde
    d_D((\bar u,\bar\mu),(u,\mu))
  \end{equation*}
  for any $(\bar u,\bar\mu)$ and $(u,\mu)$ in
  $\D$.
\end{theorem}
\begin{proof}
  Let $(y,U,h)=L(u,\mu)$ and $(\bar y,\bar U,\bar
  h)=L(\bar u,\bar\mu)$. We have
  \begin{align*}
    \norm{L(u,\mu)-L(\bar
      u,\bar\mu)}_{E_{\bar\chi}}&=\norm{(\zeta,I_{\bar\chi}^{-1}(U),h)-(\bar\zeta,I_{\bar\chi}^{-1}(\bar
      U),\bar h)}_{\bar E}\\
    &=\norm{\zeta-\bar\zeta}_{V}+\norm{I_{\bar\chi}^{-1}(U-\bar
      U)}_{H^1(\Real)\times\Real}+\norm{h-\bar h}_{L^2}\\
    &\leq\norm{\zeta-\bar\zeta}_{V}+C\norm{\Psi^{-1}I_{\bar\chi}^{-1}(U-\bar
      U)}_{H^1(\Real)\times\Real}+\norm{h-\bar h}_{L^2},
  \end{align*}
  see \eqref{eq:defPsi} for the definition of
  $\Psi$. The linear mapping $\Psi$ is continuous
  and $C$ is its operator norm, which only depends
  on $\chi$ and $\bar\chi$. Hence, since
  $I_{\chi}=I_{\bar\chi}\circ\Psi$,
  \begin{align*}
    \norm{L(u,\mu)-L(\bar u,\bar\mu)}_{E_{\bar\chi}}
    &\leq\norm{\zeta-\bar\zeta}_{V}+C\norm{I_{\chi}^{-1}(U-\bar
      U)}_{H^1(\Real)\times\Real}+\norm{h-\bar h}_{L^2}\\
    &\leq C
    \norm{(\zeta,I_{\chi}^{-1}(U),h)-(\bar\zeta,I_{\chi}^{-1}(\bar
      U),\bar h)}_{\bar E}\\
    &=C\norm{L(u,\mu)-L(\bar u,\bar\mu)}_{E_{
        \chi}}.
  \end{align*}
\end{proof}

The metric $d_\D$ on $\D$ gives the structure of a
complete metric space while it makes the semigroup
$T_t$ of conservative solutions continuous for the
Camassa--Holm equation. In that respect, it is a
suitable metric for the Camassa--Holm
equation. The definition of $d_\D$ is not
straightforward but it can be compared with more
standard topologies. We have that the mapping
\begin{equation*}
  u\mapsto(u, u_x^2\,dx)
\end{equation*}
is continuous from $H_{0,\infty}(\Real)$ into
$\mathcal{\D}$. In other words, given a sequence
$u_n\in H_{0,\infty}(\Real)$ which converges to
$u$ in $H_{0,\infty}(\Real)$, then $(u_n,
u_{n,x}^2\,dx)$ converges to $(u, u_x^2\,dx)$ in
$\mathcal{\D}$.  See \cite[Proposition
5.1]{HR}. Conversely, let $(u_n,\mu_n)$ be a
sequence in $\D$ that converges to $(u,\mu)$ in
$\D$. Then
\begin{equation*}
  u_n\rightarrow u \text{ in } L^\infty(\Real) \quad  \text{ and }\quad  \mu_n \overset{\ast}{\rightharpoonup}\mu.
\end{equation*}
See \cite[Proposition 5.2]{HR}.

\section{Conservative solutions with vanishing asymptotics}

In this section we want to clarify the connection between the approach used here in the case $c_-=c_+=0$ and the one used in \cite{HR}, which also answers the questions why the proofs are quite similar and why we speak of conservative solutions.

\begin{theorem}
 Let $(u_0,\mu_0)$ be a pair of Eulerian coordinates as in Definition~\ref{def:euler}, and $(\tilde u_0,\tilde \mu_0)$ the pair of Eulerian coordinates as defined in \cite[Definition 3.1]{HR}, such that $u_0(x)=\tilde u_0(x)$ and such that 
\begin{equation*}
\tilde\mu_0((-\infty,x))-\mu_0((-\infty,x))=\int_{-\infty}^x u_0(x)^2 dx, \quad x\in\Real.
\end{equation*}
Then the solutions $(u(t), \mu_t)$ and $(\tilde u(t), \tilde\mu_t)$ satisfy $u(t,x)=\tilde u(t,x)$ and 
 \begin{equation*}
\tilde\mu_t((-\infty,x))-\mu_t((-\infty,x))=\int_{-\infty}^x u(t,x)^2 dx, \quad x\in\Real.
\end{equation*}
\end{theorem}

\begin{proof}
Let $(u_0,\mu_0)$ be a pair of Eulerian coordinates as defined in Definition~\ref{def:euler}, and $(u_0,\tilde\mu_0)$ the set of Eulerian coordinates defined as in \cite[Definition 3.1]{HR}, such that 
\begin{equation*}
\tilde\mu_0((-\infty,x))-\mu_0((-\infty,x))=\int_{-\infty}^x u_0(x)^2 dx, \quad x\in\Real.
\end{equation*}
Moreover, let $(y,U,h)$ be the Lagrangian coordinates to $(u,\mu)$ and define 
\begin{equation*}
 \hat H(t,\xi)=\int_{-\infty}^\xi h+U^2y_\xi d\xi.
\end{equation*}
Then by construction
(cf.\ Definition~\ref{th:Ldef}), we have 
\begin{equation*}
 y(0,\xi)+\hat H(0,\xi)=\int_{-\infty}^\xi U^2y_\xi(0,\xi) d\xi+\xi.
\end{equation*}
The right-hand side belongs to $\G$, the set of
relabeling functions, which coincides with the one
used in \cite{HR}. Thus let
$f(\xi)=\int_{-\infty}^\xi U^2y_\xi (0,\xi)d\xi+\xi$.
 Then we have for almost every $\xi\in\Real$, that
\begin{equation*}
 y_0(\xi)+\mu_0((-\infty,y_0(\xi)))=\xi,
\end{equation*}
which implies 
\begin{equation*}
 y_0(\xi)+\tilde\mu_0((-\infty,y_0(\xi)))=f(\xi).
\end{equation*}
 Thus setting $\tilde y_0(\xi)=y_0(f^{-1}(\xi))$ yields 
\begin{equation*}
 \tilde y_0(\xi)+\tilde\mu_0((-\infty,\tilde y_0(\xi)))=\xi.
\end{equation*}
Moreover, one can conclude that 
\begin{equation*}
 \tilde y_0(\xi)=y_0(f^{-1}(\xi)), \quad \text{ for all } x\in\Real,
\end{equation*}
by using Definition~\ref{th:Ldef} and \cite[Theorem 3.8]{HR}.
Analogously one can proceed for the other involved functions so that $X_0\circ f^{-1}=\tilde X_0$.

Furthermore, from \eqref{eq:govsys} we can conclude that the variables $(y,U,H)$ satisfy
 \begin{subequations}
    \begin{align}
    y_t&=U,\\
    U_t&=-Q,\\
   \hat H_t&=U^3-2PU,
  \end{align}
\end{subequations}
which coincides with the system of ordinary differential equations for the
Lagrangian coordinates considered in \cite{HR}.
In addition we know from \cite[Theorem 3.7]{HR}
that $S_t(X\circ f)=S_t(X)\circ f$, and therefore
using that the mapping from Lagrangian to Eulerian
coordinates (cf.~\cite[Theorem 3.11]{HR}) is
independent of the element of the equivalence
class we choose, we obtain that the pairs
$(u_0,\mu_0)$ and $(u_0,\tilde\mu_0)$ with $\tilde
\mu_0((-\infty,x))-\mu_0((-\infty,x))=\int_{-\infty^x}u^2_0(x)dx$,
give rise to the same conservative solution.
\end{proof}

\end{document}